\documentclass[reqno]{amsart}

\usepackage{amsmath}
\usepackage{amsfonts}
\usepackage{amssymb,enumerate}
\usepackage{amsthm}
\usepackage[all]{xy}
\usepackage{rotating}
\usepackage{hyperref}

\theoremstyle{plain}
\newtheorem{lem}{Lemma}[section]
\newtheorem{cor}[lem]{Corollary}
\newtheorem{prop}[lem]{Proposition}
\newtheorem{thm}[lem]{Theorem}

\newtheorem{intthm}{Theorem}

\theoremstyle{definition}
\newtheorem{defn}[lem]{Definition}

\newtheorem{disc}[lem]{Remark}

\newtheorem{fact}[lem]{Fact}
\newtheorem{para}[lem]{}

\newtheorem*{convention*}{Convention}

\newcommand{\catd}{\mathcal{D}}

\newcommand{\pd}{\operatorname{pd}}

\newcommand{\id}{\operatorname{id}}

\newcommand{\depth}{\operatorname{depth}}

\newcommand{\edim}{\operatorname{edim}}

\newcommand{\ann}{\operatorname{Ann}}

\newcommand{\len}{\operatorname{len}}

\newcommand{\HH}{\operatorname{H}}

\newcommand{\s}{\mathfrak{S}}

\newcommand{\im}{\operatorname{Im}}
\newcommand{\shift}{\mathsf{\Sigma}}

\newcommand{\Ker}{\operatorname{Ker}}

\newcommand{\ideal}[1]{\mathfrak{#1}}
\newcommand{\m}{\ideal{m}}

\newcommand{\q}{\ideal{q}}

\newcommand{\comp}[1]{\widehat{#1}}

\newcommand{\wti}{\widetilde}

\newcommand{\bbz}{\mathbb{Z}}

\newcommand{\xra}{\xrightarrow}

\newcommand{\res}{\xra{\simeq}}

\renewcommand{\geq}{\geqslant}
\renewcommand{\leq}{\leqslant}
\renewcommand{\ker}{\Ker}
\renewcommand{\hom}{\Hom}

\newcommand{\Ext}[4][R]{\operatorname{Ext}_{#1}^{#2}(#3,#4)}	
\newcommand{\Rhom}[3][R]{\mathbf{R}\!\operatorname{Hom}_{#1}(#2,#3)}	
\newcommand{\Lotimes}[3][R]{#2\otimes^{\mathbf{L}}_{#1}#3}
\newcommand{\Otimes}[3][R]{#2\otimes_{#1}#3}
\newcommand{\Hom}[3][R]{\operatorname{Hom}_{#1}(#2,#3)}	
\newcommand{\Tor}[4][R]{\operatorname{Tor}^{#1}_{#2}(#3,#4)}

\newcommand{\spn}{\operatorname{Span}}
\newcommand{\ecodepth}{\operatorname{ecodepth}}

\numberwithin{equation}{lem}

\begin{document}

\bibliographystyle{amsplain}

\title[semidualizing complexes over Local rings of ecodepth $\leq 3$]{Local rings of embedding codepth at most 3 have only trivial semidualizing complexes}

\author{Saeed Nasseh}

\address{Department of Mathematics,
University of Nebraska--Lincoln,
Lincoln, NE  68588-0130,
USA}

\email{snasseh2@unl.edu}
\urladdr{http://www.math.unl.edu/~{}snasseh2/}

\author{Sean Sather-Wagstaff}

\address{Department of Mathematics,
North Dakota State University Dept \# 2750,
PO Box 6050,
Fargo, ND 58108-6050
USA}

\email{sean.sather-wagstaff@ndsu.edu}
\urladdr{http://www.ndsu.edu/pubweb/\~{}ssatherw/}

\thanks{Sather-Wagstaff  was supported in part by a grant from the NSA}

\date{\today}



\begin{abstract}
We prove that a local ring $R$ of embedding codepth at most 3 has 
at most two semidualizing complexes up to shift-isomorphism, namely, $R$ itself and a dualizing $R$-complex if one exists.
\end{abstract}

\maketitle


\section{Introduction} \label{sec0}

\begin{convention*}\label{conv110211a}
In this paper, $R$ is a commutative noetherian 
ring. 
In this section, assume that $(R,\m,k)$ is local.
\end{convention*}

A ``semidualizing'' $R$-complex is a homologically finite $R$-complex $X$ such that
the natural morphism $R\to\Rhom XX$ is an isomorphism in the derived category $\catd(R)$.
In particular, if $X$ is a finitely generated $R$-module, then it is semidualizing if $\Hom XX\cong R$
and $\Ext{\geq 1}XX=0$. These notions were introduced by Foxby~\cite{foxby:gmarm} and Christensen~\cite{christensen:scatac}
and, as special cases, recover Grothendieck's dualizing complexes. For some indications of their usefulness,
see, e.g., \cite{avramov:rhafgd, beck:sgidgca, sather:cidfc, sather:bnsc}.

In~\cite{nasseh:lrfsdc} we show that $R$ has only finitely many semidualizing complexes up to shift-isomorphism, 
answering a question of Vasconcelos~\cite{vasconcelos:dtmc}. The next natural question is:
how many semidualizing complexes does a given ring have up to shift-isomorphism? 
Progress on this question is limited, see, e.g., \cite{celikbas:ncdfcmlr, sather:divisor, sather:lbnsc}.
As further progress, the main result of this paper is the following, which we prove in~\ref{proof131230a}.
In the statement, the \emph{embedding codepth} of $R$ is
$\ecodepth(R)=\edim(R)-\depth(R)$.

\begin{intthm}\label{cor130816a}
Let $R$ be a local ring that  is  Golod or such that $\ecodepth(R)\leq 3$.
Then $R$ has at most two distinct semidualizing complexes up to shift isomorphism,
namely, $R$ itself and a dualizing $R$-complex if one exists.
\end{intthm}

The proof uses differential graded algebra techniques, as pioneered by Avramov and his collaborators.
It is worth noting that 
one can prove the Golod case of this result using a result of Jorgensen~\cite[Theorem 3.1]{jorgensen:gabf}.
However, our approach is different and addresses both cases simultaneously.

\subsection*{Summary}
Section~\ref{sec131230a} consists of foundational material about semidualizing objects in the DG setting.
Section~\ref{sec130815b} contains versions of results from~\cite{nasseh:oeire} for trivial extensions of DG algebras,
including Theorem~\ref{thm130816b} which is key for our proof of Theorem~\ref{cor130816a} and may be of independent interest.
Section~\ref{sec131230b} is mainly concerned with the proof of Theorem~\ref{cor130816a}.

\subsection*{Background}
We assume that the reader is familiar with many notions from the world of 
differential graded (DG) algebra. References on the subject include~\cite{apassov:hddgr, avramov:ifr, avramov:lgh, 
avramov:dgha, beck:sgidgca, felix:rht, frankild:ddgmgdga, nasseh:lrfsdc, nasseh:ldgm}. 
We most closely follow the conventions from~\cite{nasseh:lrfsdc}. 
For the reader's convenience, we specify some terminology and  notation.

Complexes of $R$-modules are indexed homologically. This includes DG algebras and DG modules. 
Also, our DG algebras are all non-negatively graded. 
Given a DG $R$-algebra $A$, the underlying graded algebra for $A$ is denoted $A^\natural$,
and $A$ is  \emph{homologically degree-wise noetherian}
if $\HH_0(A)$ is noetherian and the $\HH_0(A)$-module $\HH_i(A)$ is  finitely generated for all $i\geq 0$.
When $(R,\m)$ is local,
we say that $A$ is  \emph{local} if it is homologically degree-wise noetherian and
the ring $\HH_0(A)$ is a local $R$-algebra; in this case, the ``augmentation ideal'' of $A$ is denoted $\m_A$.

The derived category of DG $A$-modules is denoted $\catd(A)$,
and $\Lotimes[A]--$ and $\Rhom[A]--$ are the derived functors of $\Otimes[A]--$ and $\Hom[A]--$.
Given an integer $i$ and DG $A$-modules $X$ and $Y$, 
we set $\Tor[A]iXY:=\HH_i(\Lotimes[A]XY)$, and we let $\shift^iX$ denote the $i$th \emph{shift}
(or \emph{suspension}) of $X$.
Isomorphisms in $\catd(A)$ are identified by the symbol $\simeq$, and the DG modules $X$ and $Y$ are ``shift-isomorphic'',
denoted $X\sim Y$, if $X\simeq\shift^i Y$ for some $i\in\bbz$.
We write 
$\id_A(X)<\infty$ when $X$ has a bounded semi-injective resolution;
and $\pd_A(X)<\infty$ means that $X$ has a bounded  semi-free resolution.
When $A$ is local and $X$ is homologically finite, the \emph{Poincar\'e series} of $X$ is
$P^A_X(t):=\sum_{i\in\bbz}\len_k(\Tor[A]ikX)t^i$
where $k=A/\m_A$.

\section{Semidualizing DG Modules}\label{sec131230a}

\begin{convention*}
In this section, $A$ is a homologically degree-wise noetherian DG $R$-algebra.
\end{convention*}

This section contains some useful DG variations of standard results for semidualizing complexes over rings.
We begin with the following definitions  from~\cite{christensen:dvke,frankild:ddgmgdga}.

\begin{defn}\label{defn131230a}
A \emph{semidualizing} DG $A$-module is a homologically finite DG $A$-module $X$ such that
the natural homothety morphism $R\xra{\chi^X_A}\Rhom[A]XX$ is an isomorphism in $\catd(A)$.
When $A$ is a ring concentrated in degree 0,  these are ``semidualizing $A$-complexes''.
The set of shift-isomorphism classes (in $\catd(A)$) of semidualizing DG $A$-modules is denoted $\s(A)$. 
A \emph{dualizing} DG $A$-module is a semidualizing DG $A$-module $D$ such that for every homologically finite
DG $A$-module $M$
the complex $\Rhom[A]{M}{D}$ is homologically finite, and the natural morphisms
$M\to\Rhom[A]{\Rhom[A]MD}D$
and
$\Lotimes[A]DM\to \Rhom[A]{\Rhom[A]MR}D$
are isomorphisms in $\catd(A)$. The DG algebra $A$ is \emph{Gorenstein} if $A$ is dualizing.
\end{defn}

Next, we summarize some facts about dualizing DG $A$-modules. 

\begin{fact}\label{lem130821a}
Assume that $R$ has a dualizing complex $D^R$ and $A^\natural$ is finitely generated as an $R$-module.
Then $\Rhom[R]{A}{D^R}$ is a dualizing DG $A$-module by~\cite[Proposition 2.7]{frankild:ddgmgdga}.
Since $\id_R(D^R)<\infty$, we also have  $\id_A(\Rhom[R]{A}{D^R})<\infty$.
In particular, every finite dimensional DG algebra over a field has a dualizing DG module of finite injective dimension.
\end{fact}

If $A$ is concentrated in degree 0, the next result is well known, and our proof is not surprising; see, e.g., 
\cite[(2.12) Corollary]{christensen:scatac}, 
\cite[2.9.1-2]{frankild:rbsc},
and~\cite[Corollary 3.3]{gerko:sdc}.

\begin{lem}\label{lem130821b}
Assume that $A$ has a dualizing DG module $D^A$ of finite injective dimension, and let $X$ be a semidualizing DG $A$-module. 
Then $\Rhom[A]X{D^A}$ is a semidualizing DG $A$-module such that $\Lotimes[A]{X}{\Rhom[A]X{D^A}}\simeq D^A$ in $\catd(A)$. 
If $R$ and $A$ are local and either $\pd_A(X)<\infty$ or  $\id_A(A)<\infty$, then $X\sim A$.
\end{lem}

\begin{proof}
By~\cite[Theorem 1]{apassov:hddgr}, we have
\begin{align*}
\Lotimes[A]{X}{\Rhom[A]X{D^A}}
&\simeq \Rhom[A]{\Rhom[A]XX}{D^A}
\simeq \Rhom[A]{A}{D^A}
\simeq D^A.
\end{align*}
Next, we have the following commutative diagram of DG $A$-module morphisms.
$$\xymatrix@C=15mm{
A\ar[d]_-{\chi^{D^A}_A}^-\simeq\ar[r]^-{\chi^{\Rhom[A]X{D^A}}_A}&\Rhom[A]{\Rhom[A]X{D^A}}{\Rhom[A]X{D^A}}\ar[d]^-\simeq
\\
\Rhom[A]{D^A}{D^A}\ar[r]_-{\simeq}^-{\Rhom[A]{\zeta}{D^A}}&\Rhom[A]{\Lotimes[A]{X}{\Rhom[A]X{D^A}}}{D^A}
}
$$
The morphism $\zeta$ is the isomorphism from the beginning of this proof.
It follows that  $\chi^{\Rhom[A]X{D^A}}_A$ is an isomorphism, so $\Rhom[A]X{D^A}$ is a semidualizing.

Assume that $R$ and $A$ are local and either  $\pd_A(X)<\infty$ or  $\id_A(A)<\infty$.
Then by~\cite[Theorem 1]{apassov:hddgr} we have the following isomorphisms in $\catd(A)$.
$$
\Lotimes[A]{X}{\Rhom[A]XA}\simeq \Rhom[A]{\Rhom[A]XX}{A}\simeq A.
$$
It follows that $P^A_X(t)P^A_{\Rhom[A]XA}(t)=1$, so we have $P^A_X(t)=t^d$ for some $d\in\bbz$.
From this, we conclude that the minimal semi-free resolution  $F\xra\simeq X$ has $F^{\natural}\simeq\shift^d A^\natural$.
From the Leibniz rule on $F$, one concludes that $X\simeq F\simeq\shift^dA$. 
\end{proof}

The next result compares to~\cite[A.3. Lemma (a)]{christensen:dvke}.

\begin{lem}\label{lem130825a}
Let $\underline t=t_1,\ldots,t_n$ be a sequence in the Jacobson radical of $R$,
and set $K=K^R(\underline t)$. 
Let $C$ be a DG $R$-algebra such that $C^\natural$ is finitely generated over $R$, and set $B=\Otimes KC$.
Then for each homologically finite DG $C$-module $X$ we have $X\in \s(C)$ if and only if $\Lotimes[C]BX\in \s(B)$.
\end{lem}

\begin{proof}
Consider the following commutative diagram of chain maps.
$$\xymatrix{
B\ar[r]^-=\ar[d]_-{\chi^{\Lotimes[C]BX}_B}
&\Otimes KC\ar[d]^-{\Otimes[]K{\chi^X_C}}\\
\Rhom[B]{\Lotimes[C]BX}{\Lotimes[C]BX}\ar[d]_-\simeq
&\Otimes K\Rhom[C]XX\ar[dd]^-{\simeq}\\
\Rhom[C]X{\Lotimes[C] BX}\ar[d]_-= \\
\Rhom[C]X{\Lotimes[C]{(\Otimes KC)}X}\ar[r]^-\simeq
&\Rhom[C]X{\Lotimes KX}}$$
It follows that $\chi^{\Lotimes[C]BX}_B$ is an isomorphism
if and only if $\Otimes[]K{\chi^X_C}$ is an isomorphism.
Since $C$ and $\Rhom[C]XX$ are homologically degree-wise finite over $R$, we conclude (say, from a routine mapping cone
argument) that $\Otimes[]K{\chi^X_C}$ is an isomorphism
if and only if 
$\chi^X_C$ is an isomorphism.
\end{proof}

\section{Vanishing of Tor over Trivial Extensions of DG Algebras}\label{sec130815b}

\begin{convention*}
In this section, $(R,\m,k)$ is a local ring.
\end{convention*}

The point of this section is to prove that Tor-vanishing over certain trivial extensions of DG algebras implies finite projective dimension;
see Theorem~\ref{thm130816b}. We begin with a useful construction.

\begin{prop}\label{prop130820a}
Let  $(A,A_+)$ be a local DG $R$-algebra.
Let $X$ be a homologically finite DG $A$-module. Then there exists a short exact sequence
$$0\to X'\xra\alpha L\xra\pi \wti X\to 0$$
of morphisms of DG $A$-modules such that $L$ is semi-free with a finite semibasis and we have $\wti X\simeq X$
and $\im(\alpha)\subseteq A_+ L$.
\end{prop}

\begin{proof}
Let $F\xra{\simeq} X$ be a minimal semi-free resolution of $X$, and let $E$ be a semi-basis for $F$. 
Let $F^{(p)}$ be the semi-free DG $A$-submodule of $F$ spanned by $E_{\leq p}:=\bigcup_{m\leq p}E_m$. 
To be clear, ${F^{(p)}}^{\natural}$ is the graded submodule of $F^{\natural}$ spanned by $E_{\leq p}$. Note that $\partial^F({F^{(p)}}^{\natural})\subseteq {F^{(p)}}^{\natural}$. To see this, by the Leibniz rule we have $\partial^F_{i+j}(te)=\partial^A_{i}(t)e+(-1)^{i}t\partial^F_{j}(e)$ for each $t\in A$ of degree $i$ and each $e\in F$ of degree $j$. If $j\leq p$, then the term $\partial^A_{i}(t)e$ is
in $\spn_A(E_{\leq p})$ by assumption, and the term $(-1)^{i}t\partial^F_{j}(e)$ is
in $\spn_A(E_{\leq p})$   by a degree argument.

Let $s=\sup(X)$, and set $\wti X=\tau_{\leq s}(F)$, the ``soft truncation'' of $F$. Note that the natural morphism 
$F\to\wti X$ is a surjective quasiisomorphism of DG $A$-modules, so we have $\wti X\simeq F\simeq X$.
Next, set $L=F^{(s)}$, which is semi-free with a finite semibasis $E_{\leq s}$.
Furthermore, the  composition $\pi$ of natural morphisms $L=F^{(s)}\to F\to \wti X$ is surjective
because the morphism $F\to\wti X$ is surjective,  the morphism $L\to F$ is surjective in degrees $\leq s$,
and we have $\wti X_i=0$ for all $i>s$. Thus, it remains to show that $X':=\ker(\pi)\subseteq A_+L$.

By construction, for $i\leq s$, the map $\pi_s$ is an isomorphism, so $X'_i=0\subseteq A_+L$, as desired. In degree $s+1$, we have
$X_{s+1}'=\im(\partial^F_{s+1})$, which is contained in $(A_+F)_s$ by construction.
Also, since $A_+$ consists of elements of positive degree, we have 
$$(A_+F)_s\subseteq A_+\spn_A(E_{<s})\subseteq A_+F^{(s)}.$$
Thus, we have $X_{s+1}'\subseteq A_+F^{(s)}$.
Lastly, for $i>s$, we have $X'_i=F^{(s)}_i=(A_+F^{(s)})_i$ by a degree argument, since $F^{(s)}$ is 
generated over $A$ in degrees $\leq s$.
\end{proof}

\begin{cor}\label{cor130820a}
Let $(A,A_+)$ be a bounded local DG $R$-algebra.
Let $X$ and $Y$ be  homologically finite DG $A$-modules of infinite projective dimension over $A$. 
Assume that $\Tor[A]{i\gg 0}{X}{Y}=0$.
Then there are homologically finite DG $A$-modules $X',Y'$ of infinite projective dimension over $A$
such that $\ann_A(A_+)X'=0=\ann_A(A_+)Y'$
and $\Tor[A]{i\gg 0}{X'}{Y'}=0$.
\end{cor}

\begin{proof}
By Proposition~\ref{prop130820a}, there are short exact sequences
\begin{gather*}
0\to X'\xra\alpha L\to \wti X\to 0 \\
0\to Y'\xra\beta M\to \wti Y\to 0 
\end{gather*}
of morphisms of DG $A$-modules such that $L$ and $M$ are semi-free with a finite semibases, and $\wti X\simeq X$,
$\wti Y\simeq Y$,
$\im(\alpha)\subseteq A_+ L$,
and $\im(\beta)\subseteq A_+ M$.
The condition $\im(\alpha)\subseteq A_+ L$ implies that
$$\ann_A(A_+)X'\cong \ann_A(A_+)\im(\alpha)\subseteq \ann_A(A_+)A_+ L=0$$
and similarly  $\ann_A(A_+)Y'=0$.
By assumption, $\pd_A(L)<\infty$, so the condition $\pd_A(\wti X)=\pd_A(X)=\infty$ implies that $\pd_A(X')=\infty$,
because of the exact sequence $0\to X'\to L\to \wti X\to 0$. 
And similarly  $\pd_A(Y')=\infty$.
Finally, since $A$ is bounded and $\pd_A(L)$ is finite, the fact that $Y$ is homologically bounded implies that
$\Tor[A]{i\gg 0}{L}{Y}=0$.
By assumption, we have $\Tor[A]{i\gg 0}{\wti X}{Y}\cong\Tor[A]{i\gg 0}{X}{Y}=0$
so the above exact sequence implies that $\Tor[A]{i\gg 0}{X'}{Y}=0$.
Similarly, we deduce $\Tor[A]{i\gg 0}{X'}{Y'}=0$, as desired.
\end{proof}

Compare the next two results to~\cite[Lemma 3.2 and Theorem 3.1]{nasseh:oeire}.

\begin{lem}\label{lem130815a}
Let 
$(B,\m_B,k)$ be a local DG $R$-algebra.
Set $A=B\ltimes \shift^n k$ for some $n\geq 0$,
and let $\m_A$ be the augmentation ideal of $A$. Let $x$ be a generator for the DG ideal $0\oplus \shift^n k\subseteq A$.
Let $X$, $Y$ be DG $B$-modules, i.e., DG $A$-modules such that $xX=0=xY$.
Then for all $i\in\bbz$ we have $R$-module isomorphisms
$$\Tor[A]{i}XY\cong \Tor[B]{i}XY\bigoplus\left(\bigoplus_{p+q=i-n-1}\Otimes[k]{\Tor[A]{p}Xk}{\Tor[B]{q}kY}\right).$$
\end{lem}

\begin{proof}
Let $s\colon L\res B$ be a semi-free resolution over $A$.
Let $p\colon A\to B$ be the natural surjection, and let $\wti p\colon A\to L$ be a lift of $p$.
Hence the following diagram of morphisms of DG $A$-modules
$$\xymatrix{
A\ar[r]^-{\wti p}\ar[rd]_-p &L\ar[d]^<<<s_<<<\simeq \\
&B}$$
commutes up to homotopy.
Apply $\Otimes[A]{X}{-}$ to obtain the next diagram of morphisms of DG $B$-modules
$$\xymatrix{
\Otimes[A]{X}{A}\ar[r]^-{\Otimes[]{X}{\wti p}}\ar[rd]_-{\Otimes[]{X}{p}}^-\cong &\Otimes[A]{X}{L}\ar[d]^{\Otimes[]{X}{s}} \\
&\Otimes[A]{X}{B}}$$
that commutes up to homotopy. Note that $\Otimes[]{X}{p}$ is an isomorphism since $xX=0$.
Also, the chain map $\Otimes[A]{X}{\wti p}$ represents the morphism $\Lotimes[A]{X}{p}\colon\Lotimes[A]{X}{A}\to\Lotimes[A]{X}{B}$
in $\catd(B)$. It follows that $\Lotimes[A]{X}{p}$ has a left-inverse in $\catd(B)$,
so we have the first isomorphism (in $\catd(B)$) in the next sequence:
\begin{align*}
\Lotimes[A]{X}{B}
&\simeq\left(\Lotimes[A]{X}{A}\right)\bigoplus\left(\Lotimes[A]{X}{(\shift xA)}\right) \\
&\simeq\left(\Lotimes[A]{X}{A}\right)\bigoplus\left(\Lotimes[A]{X}{(\shift^{n+1}k)}\right)\\
&\simeq\left(\Lotimes[A]{X}{A}\right)\bigoplus\left(\shift^{n+1}\Lotimes[A]{X}{k}\right).
\end{align*}
The second isomorphism is from the assumption $\m_Ax=0$.
(Note that this isomorphism is in $\catd(A)$, hence also in $\catd(B)$ via the split injection $B\to A$.)
Now, we apply $\Lotimes[B]{-}{Y}$ to these isomorphisms to conclude that
\begin{align*}
\Lotimes[A]{X}{Y}
&\simeq\Lotimes[B]{(\Lotimes[A]{X}{B})}{Y}\\
&\simeq\left(\Lotimes[B]{(\Lotimes[A]{X}{A})}{Y}\right)\bigoplus\left(\shift^{n+1}\Lotimes[B]{(\Lotimes[A]{X}{k})}{Y}\right)\\
&\simeq\left(\Lotimes[B]{X}{Y}\right)\bigoplus\left(\shift^{n+1}\Lotimes[k]{(\Lotimes[A]{X}{k})}{(\Lotimes[B]{k}{Y})}\right).
\end{align*}
Apply $\HH_i(-)$ to obtain the first isomorphism in the next sequence:
\begin{align*}
\Tor[A]{i}{X}{Y}
&\cong\Tor[B]{i}{X}{Y}\bigoplus\HH_i\left(\shift^{n+1}\Lotimes[k]{(\Lotimes[A]{X}{k})}{(\Lotimes[B]{k}{Y})}\right)\\
&\cong\Tor[B]{i}{X}{Y}\bigoplus\HH_{i-n-1}\left(\Lotimes[k]{(\Lotimes[A]{X}{k})}{(\Lotimes[B]{k}{Y})}\right)\\
&\cong\Tor[B]{i}{X}{Y}\bigoplus\left(\bigoplus_{p+q=i-n-1}\Otimes[k]{\HH_p(\Lotimes[A]{X}{k})}{\HH_q(\Lotimes[B]{k}{Y})}\right)\\
&\cong\Tor[B]{i}{X}{Y}\bigoplus\left(\bigoplus_{p+q=i-n-1}\Otimes[k]{\Tor[A]p{X}{k}}{\Tor[B]q{k}{Y}}\right).
\end{align*}
The third isomorphism is from the K\"unneth formula.
\end{proof}

\begin{thm}\label{thm130816b}
Let 
$(B,B_+,k)$ be a bounded local  DG $R$-algebra.
Set $A=B\ltimes \shift^n k$ for some $n\geq 1$. Let $x$ be a generator for the DG ideal $0\oplus \shift^n k\subseteq A$.
Let $X$ and $Y$ be non-zero homologically finite 
DG $A$-modules such that $\Tor[A]{i\gg 0}XY=0$. Then $\pd_A(X)<\infty$ or $\pd_A(Y)<\infty$.
\end{thm}

\begin{proof}
Assume by way of contradiction that $\pd_A(X)=\infty=\pd_A(Y)$.
Thus, Corollary~\ref{cor130820a}
provides homologically finite DG $A$-modules $X',Y'$ of infinite projective dimension
such that $\ann_A(A_+)X'=0=\ann_A(A_+)Y'$
and such that $\Tor[A]{i\gg 0}{X'}{Y'}=0$.
Thus, we may replace $X$ and $Y$ by $X'$ and $Y'$ to assume that $\ann_A(A_+)X=0=\ann_A(A_+)Y$.
In particular, we have $xX=0=xY$.

Since $\Tor[A]{i\gg 0}XY=0$,
Lemma~\ref{lem130815a} implies that
\begin{equation}\label{eq131227a}
\bigoplus_{p+q=i-n-1}\Otimes[k]{\Tor[A]{p}Xk}{\Tor[B]{q}kY}=0
\end{equation}
for all $i\gg 0$.
The fact that $\pd_A(Y)$ is infinite implies that $Y\not\simeq 0$.
Since $Y$ is homologically finite, Nakayama's Lemma implies that there is an integer $q_0$ such that
$\Tor[B]{q_0}kY\neq 0$.
Thus, equation~\eqref{eq131227a} for $i\gg 0$ implies that $\Tor[A]{p\gg 0}Xk=0$,
contradicting the assumption $\pd_A(X)=\infty$.
\end{proof}

\section{The Number Of Semidualizing Complexes For Embedding Codepth 3}
\label{sec131230b}

\begin{convention*}
In this section, $(R,\m,k)$ is a local ring.
\end{convention*}

The purpose of this section is to prove Theorem~\ref{cor130816a} from the introduction, which we do in~\ref{proof131230a}.
Our main tool for this is the following result.

\begin{thm}\label{cor130821a}
Assume that $R$ admits a dualizing complex, and let 
$(B,B_+,k)$ be a bounded homologically finite local  DG $R$-algebra.
Set $A=B\ltimes W$ for some
non-zero finitely generated positively graded $k$-vector space $W$. 
Then $|\s(A)|\leq 2$.
\end{thm}

\begin{proof}
Fact~\ref{lem130821a} implies that $A$ has a dualizing DG module $D^A$ of finite injective dimension.
Set $(-)^{\dagger}=\Rhom[A]{-}{D^A}$.
It suffices to show that for every semidualizing DG $A$-module $X$ we have $X\sim A$ or $X\sim D^A$.

Case 1: $W=\shift^n k$ for some $n\geq 1$. The isomorphism $\Lotimes[A]{X}{X^{\dagger}}\simeq D^A$
from Lemma~\ref{lem130821b} implies that $\Tor[A]{i\gg 0}{X}{X^{\dagger}}=0$. By Theorem~\ref{thm130816b} either $\pd_A(X)<\infty$ or $\pd_A(X^{\dagger})<\infty$. And by Lemma~\ref{lem130821b} we have $X\sim A$ or $X^{\dagger}\sim A$. If $X^{\dagger}\sim A$, then by 
definition of $D^A$
we have
$X\simeq X^{\dagger \dagger}\sim A^{\dagger}\simeq D^A$, as desired.

Case 2: General case.
Since $W\neq 0$, write $W = W' \oplus \shift^n k$ for some $n\geq 1$, and set $B' := B \ltimes W'$. It follows that $A \cong B' \ltimes \shift^n k$, 
so  the assertion follows from the previous case.
\end{proof}

\begin{disc}\label{disc131221a}
Let $\underline{t}=t_1,\cdots,t_n$ be a minimal generating set for $\mathfrak m$,
and consider the Koszul complex $K=K^R(\underline t)$.
From~\cite[(1.2)]{avramov:cslrec3} (or~\cite[(2.8)]{avramov:lgh}) there is a  finite-dimensional DG $k$-algebra $A$ that is
linked to $K$ by a sequence of quasiisomorphisms of DG algebras. 
As in~\cite[5.4 (Proof of Theorem A)]{nasseh:lrfsdc}, one has an injection $\s(R)\hookrightarrow\s(\comp R)$
and  bijections $\s(\comp R)\xra\cong\s(K)\xra\cong\s(A)$.

By definition
(or by a result of Golod~\cite{golod:hslr}), $R$ is Golod if and only if $A$ is of the form $k\ltimes W$ for some
finitely generated positively graded $k$-vector space $W$. 

Assume for the rest of this remark that $c=\ecodepth(R)\leq 3$.
Up to isomorphism, 
the algebra $A$ has $\partial^A=0$ and is in one of the classes described in the next table, copied from~\cite[1.3]{avramov:cslrec3}.
\begin{center}
\begin{tabular}{lllll}
\hline
Class & $A$&$B$&$C$&$D$ \\
\hline
$\mathbf{C}(c)$&$B$&$\bigwedge_k\shift k^c$\\
$\mathbf{S}$&$B\ltimes W$&$k$\\
$\mathbf{T}$&$B\ltimes W$&$C\ltimes\shift(C/C_{\geq 2})$&$\bigwedge_k\shift k^2$\\
$\mathbf{B}$&$B\ltimes W$&$C\ltimes\shift C_+$&$\bigwedge_k\shift k^2$\\
$\mathbf{G}(r)$&$B\ltimes W$&$C\ltimes\Hom[k]{C}{\shift^3k}$&$k\ltimes\shift k^r$\\
$\mathbf{H}(p,q)$&$B\ltimes W$&$\Otimes[k]CD$&$k\ltimes(\shift k^p\oplus\shift^2k^q)$&$k\ltimes\shift k$
\end{tabular}
\end{center}

\

\noindent Here, $W$ is a finitely generated positively graded $k$-vector space such that $B_+W=0$.
Note that if $R$ is not regular, then  $\sum_i(-1)^i\len_k(A_i)=\sum_i(-1)^i\len_k(\HH_i(K))=0$.
\end{disc}

\begin{para}[Proof of Theorem~\ref{cor130816a}] \label{proof131230a}
Continue with the notation from Remark~\ref{disc131221a}.
Recall that $\comp R$ has a dualizing complex $D^{\comp R}$, and that $R$ is Gorenstein if and only if $D^{\comp R}\sim \comp R$.
Hence, it suffices to show that $|\s(A)|\leq 2$.

Assume for this paragraph that $R$ is Golod. Then  we have
$A\cong k\ltimes W$ for some
finitely generated positively graded $k$-vector space $W$. 
If $W=0$, then $A\cong k$ which is a commutative local Gorenstein ring, hence $|\s(A)|=1$ in this case.
If $W\neq 0$, then $|\s(A)|\leq 2$ by Theorem~\ref{cor130821a}.

Assume for the rest of the proof that $c=\ecodepth(R)\leq 3$. We analyze the classes from Remark~\ref{disc131221a}.
Note that if $R$ is  Gorenstein, then $|\s(R)|=1$.
Also, if $W\neq 0$, then the conclusion $|\s(A)|\leq 2$ follows from Theorem~\ref{cor130821a}.
Thus, we assume for the rest of the proof that $W=0$.

If $R$ is in the class $\mathbf{C}(c)$, then $R$ is a complete intersection (hence Gorenstein), which has already been treated.
If $R$ is in the class $\mathbf{S}$, then $R$ is Golod, which has also already been treated.

(Class \textbf{T}) 
In this case the algebra $C$ has the form $0\to k\to k^2\to k\to 0$. In particular, $C_{\geq 2}=\shift^2k$. 
It follows that
$$\sum_i(-1)^i\len_k(B_i)=\sum_{i}(-1)^i\len_k(C_i)+\sum_{i<2}(-1)^{i-1}\len_k(C_i)= 1\neq 0.$$
Since we know that 
\begin{align*}
0&=\sum_i(-1)^i\len_k(A_i)\\
&=\sum_i(-1)^i\len_k(B_i)+\sum_i(-1)^i\len_k(W_i)\\
&=1+\sum_i(-1)^i\len_k(W_i)
\end{align*}
we have $W\neq 0$, which is a case we have already treated.

(Class \textbf{B}) 
As in the previous case, one has $W\neq 0$, which has already been~treated.

(Class \textbf{G}(r)) 
The DG $C$-module $\hom[k]{C}{\shift^3k}$ is dualizing, so the algebra 
$$A\cong B=C\ltimes\Hom[k]{C}{\shift^3k}$$ 
is Gorenstein 
by~\cite[Theorem 2.2]{jorgensen:rdc}.
It follows from~\cite[Theorem III]{frankild:ddgmgdga} that $A\simeq \Hom[k]{A}{\shift^3k}$.
Since $\Hom[k]{A}{\shift^3k}$ is bounded and semi-injective over $A$, we conclude that $\id_A(A)<\infty$.
Thus, we have $|\s(A)|=1$ by Lemma~\ref{lem130821b}.

(Class \textbf{H}(p,q)) 
We first note that $|\s(C)|\leq 2$. Indeed, if $p=q=0$, then we have $C=k$ so $|\s(C)|=1$, as above;
if $p\neq 0$ or $\q\neq 0$, then this follows from Theorem~\ref{cor130821a}. 
Hence, it remains to show that $|\s(B)|=|\s(C)|$.
For this, 
consider the map $\s(C)\to \s(B)$ defined by $X\mapsto\Lotimes[C]BX$. This is well-defined by Lemma~\ref{lem130825a},
and it is bijective by~\cite[Theorems 3.4 and 3.11]{nasseh:ldgm}. 
(This uses the fact that $D$ is isomorphic to the trivial Koszul complex $K^k(0)$.)
\qed
\end{para}

\section*{Acknowledgments}
We are grateful to Lucho Avramov for helpful conversations about this material.

\providecommand{\bysame}{\leavevmode\hbox to3em{\hrulefill}\thinspace}
\providecommand{\MR}{\relax\ifhmode\unskip\space\fi MR }
\providecommand{\MRhref}[2]{%
  \href{http://www.ams.org/mathscinet-getitem?mr=#1}{#2}
}
\providecommand{\href}[2]{#2}

\end{document}